\documentclass{amsart}%{{{1
\newif\ifdraft
\draftfalse

\ifdraft\usepackage[notref,notcite]{showkeys}\fi
\usepackage{graphicx,subfigure,suffix}
\usepackage{amssymb,amsmath,amsfonts,amssymb,mathtools}%numbysec}
\usepackage{latexsym}
\usepackage{cite}
\usepackage[pdfborder={0 0 .1}]{hyperref}

\def \Tm {\mathbb T}

\newcommand{\eps}{\varepsilon}

\allowdisplaybreaks \numberwithin{equation}{section}

\newtheorem{theorem}{Theorem}[section]
\newtheorem{conjecture}[theorem]{Conjecture}
\newtheorem{lemma}[theorem]{Lemma}
\newtheorem{proposition}[theorem]{Proposition}

\newtheorem*{theorem*}{Theorem}
\newtheorem*{lemma*}{Lemma}
\newtheorem*{proposition*}{Proposition}
\newtheorem*{corollary*}{Corollary}

\theoremstyle{definition}
\newtheorem{definition}[theorem]{Definition}

\theoremstyle{remark}
\newtheorem{remark}[theorem]{Remark}
\newtheorem*{remark*}{Remark}

\newcommand{\Chi}[1]{\Chi*{\{#1\}}}
\WithSuffix\newcommand\Chi*[1]{\chi_{\raise-.5ex\hbox{$\scriptstyle#1$}}}
\DeclarePairedDelimiter{\paren}{(}{)}

\DeclarePairedDelimiter{\set}{\{}{\}}
\DeclarePairedDelimiter{\abs}{\lvert}{\rvert}
\DeclarePairedDelimiter{\norm}{\lVert}{\rVert}
\newcommand{\delt}{\partial_t}
\newcommand{\del}{\partial}
\newcommand{\grad}{\nabla}
\renewcommand{\leq}{\leqslant}
\renewcommand{\geq}{\geqslant}

\newcommand{\T}{\mathbb T}

\newcommand{\defeq}{\stackrel{\scriptscriptstyle\text{def}}{=}}
\DeclareMathOperator{\Lip}{Lip}
\DeclareMathOperator{\supp}{supp}
\newcommand{\inv}{^{-1}}
\newcommand{\lap}{\Delta}

\providecommand{\cites}[1]{\cite{#1}}

\makeatletter
\@mparswitchfalse%Put all sidenotes on the right.
\makeatother

\newcommand{\sidenote}[1]{%
  \ifdraft
    \marginpar[\raggedleft\tiny #1]{\raggedright\tiny #1}%
  \fi%
}

\begin{document}
  %{{{1 Title stuff

  \title%
    [\ifdraft DRAFT: \fi Lower bounds on the mix norm of passive scalars]
    {\ifdraft DRAFT: \fi Lower bounds on the mix norm of passive scalars advected by incompressible enstrophy-constrained flows.}
  \begin{abstract}
    %\sidenote{Mon 07/22 GI: Revisit after paper is written.}
    Consider a diffusion-free passive scalar $\theta$ being mixed by an incompressible flow $u$ on the torus $\Tm^d$.
    %Suppose that there is a constraint on the vector field expressed in terms of bounds on $L^2$ norm of $\nabla u,$ such as $\int_0^T \int |\nabla u|^2\,dx \,dt \leq C$.
    %Such bounds are natural and are satisfied by solutions of Navier-Stokes equations.
    Our aim is to study how well this scalar can be mixed under an enstrophy constraint  on the advecting velocity field.
    %Various notions and measures of the degree of mixing have been used to quantify the mixing process.
    %One of the most natural measures, is the decay in time of negative Sobolev norms. In this note, we prove rigorous
    Our main result shows that the mix-norm ($\norm{\theta(t)}_{H^{-1}}$) is bounded below by an exponential function of time.
    The exponential decay rate we obtain is not universal and depends on the size of the support of the initial data.
    We also perform numerical simulations and confirm that the numerically observed decay rate scales similarly to the rigorous lower bound, at least for a significant initial period of time.
    %Finally, we also present a scaling argument suggesting that if an exponential lower bound exists with a universal decay rate, then under no slip boundary conditions on the advecting fluid, the decay rate must in fact be %algebraic.
    %lower bounds on $\|\theta(\cdot, t)\|_{H^{-1}}$ which are exponentially decaying in time.
    %Such behavior has been suggested by earlier numerical experiments by Doering, Lin and Thiffeault \cite{bblLinThiffeaultDoering11}. The proof is based on a connection with notion of mixing
    The main idea behind our proof is to use recent work of Crippa and DeLellis ('08) making progress towards the resolution of Bressan's rearrangement cost conjecture.
  \end{abstract}

  \author{Gautam Iyer}
  \address{\hskip-\parindent
  Gautam Iyer\\
  Department of Mathematics\\
  Carnegie Mellon University \\
  %       Street address or POBox\\
  Pittsburgh, PA 15213}
  \email{gautam@math.cmu.edu}

  \author{Alexander Kiselev}
  \address{\hskip-\parindent
  Alexander Kiselev\\
  Department of Mathematics\\
  University of Wisconsin-Madison\\
  %        Street address or POBox\\
  Ma\-dison, WI 53706, USA}
  \email{kiselev@math.wisc.edu}

  \author{Xiaoqian Xu}
  \address{\hskip-\parindent
  Xiaoqian Xu\\
  Department of Mathematics\\
  University of Wisconsin-Madison\\
  %        Street address or POBox\\
  Madison, WI 53706, USA}
  \email{xxu@math.wisc.edu}

  \thanks{This material is based upon work partially supported by the National Science Foundation under grants
  DMS-1007914, % Gautam
  DMS-1104415, % Sasha/Xiao
  DMS-1159133, % Sasha/Xiao
  DMS-1252912. % Gautam
  GI acknowledges partial support from an Alfred P. Sloan research fellowship. AK acknowledges partial support from a Guggenheim fellowship.
  The authors also thank the Center for Nonlinear Analysis (NSF Grants No. DMS-0405343 and DMS-0635983), where part of this research was carried out. }

  \maketitle

  \section{Introduction}%{{{1

  The mixing of tracer particles by fluid flows is ubiquitous in nature, and have applications ranging from weather forecasting to food processing.
  An important question that has attracted attention recently is to study ``how well'' tracers can be mixed under a constraint on the advecting velocity field, and what is the optimal choice of the ``best mixing'' velocity field (see~\cite{bblThiffeault12} for a recent review).

  Our aim in this paper is to study how well passive tracers can be mixed under an \emph{enstrophy constraint} on the advecting fluid.
  By passive, we mean that the tracers provide no feedback to the advecting velocity field.
  Further, we assume that diffusion of the tracer particles is weak and can be neglected on the relevant time scales.
  Mathematically, the density of such tracers (known as passive scalars) is modeled by the transport equation
  \begin{equation}\label{eq:1}
    \partial_t\theta(x,t)+u\cdot \nabla\theta=0, \,\,\,\theta(x,0) = \theta_0(x).
  \end{equation}
  To model stirring, the advecting velocity field~$u$ is assumed to be incompressible.
  For simplicity we study~\eqref{eq:1} with periodic boundary conditions (with period $1$), mean zero initial data, and assume that all functions in question are smooth.

  %There is a significant body of literature devoted to the study of mixing~[TODO cite], and we only provide a brief introduction to the results that are most relevant to this paper.
  %First, following
  The first step is to quantify ``how well'' a passive scalar is mixed in our context. For \emph{diffusive} passive scalars, the decay of the variance is a commonly used measure of mixing (see for instance \cites{bblConstantinKiselevRyzhikZlatos08,bblDoeringThiffeault06,bblThiffeaultDoeringGibbon04,bblShawThiffeaultDoering07} and references there in).
  But for diffusion free scalars the variance is a conserved and does not change with time.
  Thus, following~\cite{bblLinThiffeaultDoering11} we quantify mixing using the $H^{-1}$-Sobolev norm: the \emph{smaller} $\norm{\theta}_{H^{-1}}$, the better mixed the scalar $\theta$ is.

  The reason for using a negative Sobolev norm in this context has its roots in~\cites{bblDoeringThiffeault06,bblLinThiffeaultDoering11,bblMathewMezicPetzold05,bblShawThiffeaultDoering07}.
  The motivation is that if the flow generated by the velocity field is mixing in the ergodic theory sense, then any advected quantity (in particular $\theta$) converges to $0$ weakly in $L^2$ as $t \to \infty$.
  This can be shown to imply that $\norm{\theta(\cdot, t)}_{H^{s}} \to 0$ for all $s < 0$, and conversely, if $\norm{\theta(\cdot, t)}_{H^{s}} \to 0$ for some $s < 0$ then $\theta(x,t)$ converges
  weakly to zero. Thus any negative Sobolev norm of $\theta$ can in principle be used to quantify its mixing properties.
  In two dimensions the choice of using the $H^{-1}$ norm in particular was suggested by Lin et.\ al.~\cite{bblLinThiffeaultDoering11} as it scales like the area dominant unmixed regions; a natural length scale associated with the system.
  We will work with the same Sobolev norm in any dimension $d$; the ratio of $H^{-1}$ norm to $L^2$ norm has a dimension of length, and since the $L^2$ norm of $\theta(x,t)$ is conserved, the $H^{-1}$ norm provides a natural length scale associated with the mixing process.

  The questions we study in this paper are motivated by recent work of Lin et.\ al.~\cite{bblLinThiffeaultDoering11}.
  In~\cite{bblLinThiffeaultDoering11}, the authors address two questions on the two dimensional torus:
  \begin{itemize}
    \item
      The time decay of $\norm{\theta(t)}_{H^{-1}}$, given the \emph{fixed energy} constraint $\norm{u(t)}_{L^2} = U$.

    \item
      The time decay of $\norm{\theta(t)}_{H^{-1}}$ given a \emph{fixed enstrophy} constraint of the form $\norm{\grad u(t)}_{L^2} = F$.
  \end{itemize}

  In the first case the authors prove a lower bound for~$\norm{\theta(\cdot, t)}_{H^{-1}(\T^2)}$ that is linear in $t$, with negative slope.
  This suggests that it may be possible to ``mix perfectly in finite time''; namely choose $u$ in a manner that drives $\norm{\theta(\cdot, t)}_{H^{-1}}$ to zero in finite time.
  This was followed by an explicit example in~\cite{bblLunasinLinNovikovMazzucatoDoering12} exhibiting finite time perfect mixing, under a finite energy constraint.
  This example uses an elegant ``slice and dice'' construction, which requires the advecting velocity field to develop finer and finer scales.
  Thus, while their example maintains a fixed energy constraint, the enstrophy ($\norm{\grad u}_{L^2}$) explodes.
  Together with the numerical analysis in~\cites{bblLinThiffeaultDoering11,bblLunasinLinNovikovMazzucatoDoering12} this suggests that finite time perfect mixing by an enstrophy constrained incompressible flow might be impossible.
  Our main theorem settles this affirmatively.
  \begin{theorem}\label{thmMain}%{{{
    Let $u$ be a smooth (time dependent) incompressible periodic vector field on the $d$-dimensional torus, and let $\theta$ solve~\eqref{eq:1} with periodic boundary conditions and $L^\infty$ initial data $\theta_0$.
    For any $p > 1$ and $\lambda \in (0,1)$ there exists a length scale $r_0 = r_0( \theta_0, \lambda),$ an explicit constant $\eps_0 = \eps_0(\lambda, d),$ and a constant $c = c(d, p)$ such that
    \begin{equation}\label{eq:solve}
      \norm{\theta(t)}_{H^{-1}}
	\geq \eps_0 r_0^{d/2+1} \norm{\theta_0}_{L^\infty} \exp \paren[\Big]{ \frac{-c}{m \paren{ A_\lambda }^{1/p} } \int_0^t \norm{\grad u(s)}_{L^p} \, ds }.
    \end{equation}
    Here $A_\lambda$ is the super level set $\{ \theta_0 > \lambda \norm{\theta_0}_{L^\infty} \}$.

    In particular, if the instantaneous enstrophy constraint $\norm{\grad u}_{L^2} \leq F$ is enforced, then $\norm{\theta(t)}_{H^{-1}}$ decays at most exponentially with time.
  \end{theorem}%}}}

  Before commenting on the $r_0$ and $m(A_\lambda)$ dependence, we briefly mention some applications.
  There are many physical situations where $\int_0^t \norm{\grad u(s)}_{L^2} \,ds$ is well controlled.
  Some examples are when $u$ satisfies the incompressible Navier-Stokes equations with $\dot H^{-1}$ forcing~\cite{bblDoeringGibbon95,bblConstantinFoias88}, the 2D incompressible Euler equations~\cite{bblBardosTiti07} or a variety of active scalar equations including the critical surface quasi-geostrophic equation~\cite{bblKiselevNazarovVolberg07,bblConstantinVicol12,bblCaffarelliVasseur10,bblChaeConstantinCordobaGancedoWu12}.
  %\sidenote{Mon 07/22 GI: I deleted the reference to Wirosoetisno. He has acknowledged the error by email, and communicated to the editor; I assume a retraction or erratum will appear. If you want to point out the error explicitly, it would be nice to ask him first.}
  In each of these situations the passive scalars can not be mixed perfectly in finite time.
  More precisely, a lower bound for the $H^{-1}$-norm of the scalar density can be read off using~\eqref{eq:solve} and the appropriate control on $\norm{\grad u}_{L^2}$.

  %\sidenote{Mon 07/22 GI: Some of these ideas were suggested in~\cite{bblLunasinLinNovikovMazzucatoDoering12}, and we should mention it somewhere.}
  We also mention that the proof of this theorem is not based on energy methods.
  Instead, the main idea is to relate the notion of ``mixed to scale~$\delta$'' to the $H^{-1}$ norm, and use recent progress by Crippa and DeLellis~\cite{bblCrippaDeLellis08} towards Bressan's rearrangement cost conjecture~\cite{bblBressan03}.
  Some of these ideas were already suggested in~\cite{bblLunasinLinNovikovMazzucatoDoering12}.
  \medskip

  We defer the proof of Theorem~\ref{thmMain} to Section~\ref{sxnMainTheorem}, and
  %Before turning to the proof we
  pause to analyze the dependence of the bound in \eqref{eq:solve} on $r_0$ and $m(A_\lambda)$.

  The length scale $r_0$ is morally the scale at which the super level set $A_\lambda$ is ``unmixed''; a notion that is made precise later.
  Our proof, however, imposes a slightly stronger condition: namely, our proof will show that $r_0$ can be any length scale such that ``most'' of the super level set $A_\lambda$ occupies ``most'' of the union of disjoint balls of radius \emph{at least} $r_0$.
  While we are presently unable to estimate $r_0$ in terms of a tangible norm of $\theta_0$, we remark that we at least expect a connection between $r_0$ and the ratio of the measure of $A_\lambda$ to the perimeter of $A_\lambda$ (see~\cite{bblSlepcev08} for a related notion).

  On the other hand, we point out that the pre-factor in~\eqref{eq:solve} can be improved at the expense of the decay rate.
  To see this, suppose for some $\kappa \in [0, 1/2)$ there exists $N$ disjoint balls of radius at least $r_1$ such that the fraction of each of these balls occupied by $A_\lambda$ is \emph{at least} $1 - \kappa$. Then our proof will show that~\eqref{eq:solve} in Theorem~\ref{thmMain} can be replaced by
    \begin{gather}\label{eqnSolvePrime}
      \tag{\ref{eq:solve}$'$}
      \norm{\theta(t)}_{H^{-1}}
	\geq \eps_0 r_1^{d/2+1} \norm{\theta_0}_{L^\infty} \exp \paren[\Big]{ \frac{-c}{\paren{ N r_1^d }^{1/p} } \int_0^t \norm{\grad u(s)}_{L^p} \, ds }.
    \end{gather}
    In this case, if $\theta_0 \in C^1$, the mean value theorem will guarantee that we can choose $N = 1$ and $r_1 > \frac{\norm{\theta_0}_{L^\infty}}{C \norm{\grad \theta_0}_{L^\infty}}$ for a purely dimensional constant $C$.

  Next we turn to the exponential decay rate.
  The dependence of this on $m(A_\lambda)$ is natural.
  To see this, suppose momentarily that $\theta_0$ only takes on the values $\pm1$ or $0$ representing two insoluble, immiscible fluids which are injected into a large  fluid container.
  Physical intuition suggests that the less the amount of fluid that is injected, the faster one can mix it.
  Indeed, this is reflected in~\eqref{eq:solve} as in this case $m(A_\lambda) = \frac{1}{2} m( \supp( \theta_0 ) )$;
    so the smaller the support of the initial data, the worse the lower bound~\eqref{eq:solve} is.
  We mention that a bound similar to~\eqref{eq:solve} was proved in~\cite{bblSeis} using optimal transport and ideas from~\cite{bblBrenierOttoSeis}.
  In~\cite{bblSeis}, however, the author only considers bounded variation ``binary phase'' initial data, where the two phases occupy the entire region;
  consequently the result does not capture the dependence of the decay rate on the initial data.
  %We stress, however, that our exponential decay rate lower bound is not universal, in the sense that it depends on $\theta_0$ and, as one can intuitively expect, weakens if the support of the initial data is small.

  %For bounded variation initial data whose range is $\{\pm 1\}$ it was recently shown in~\cite{bblSeis} that it is possible to prove an exponential lower bound with a universal decay rate.
  %For $L^\infty$ initial data which is allowed to take on arbitrary values in $\mathbb R$, the universality question remains open.

 % In fact, our numerical simulations show that the decay rate may depend strongly on the data for a significant initial period of time.
 % Overall, the rough picture that seems to emerge is that the decay rate should be universal for spatially homogeneous data that is evenly spread over the whole region.
 % For spatially inhomogeneous or sparse data there can be a long period of time where decay rate is data dependent.
 % Whether at later times decay must revert to slower universal
 % rate is a very interesting question.  We elaborate on this in Section~\ref{sxnNumerics} and present numerical evidence.

  %We must note that
  We do not know if the estimate of the exponential decay provided by bound~\eqref{eq:solve} is optimal, however, numerical simulations suggest that it may be not far off.
  A good candidate for the velocity field that might achieve the optimal lower bound was presented in~\cite{bblLinThiffeaultDoering11} using a steepest descent method (equation~\eqref{eqnVelocity}, below).
  Due to the nonlinear nature of this formula, it is hard to rigorously prove upper bounds; but all our numerical simulations in Section~\ref{sxnNumerics} seem to indicate an exponential lower bound with a decay rate that is in a good qualitative agreement with~\eqref{eq:solve}.

 % Physical intuition suggests ``the more there is, the harder it is to mix'', so some dependence of the exponential decay rate on $m(\supp(\theta_0))$ is not surprising.
 % If all the fluid motion is localized to small region, it can likely be more efficient.
 % This is reflected in the lower bound \eqref{eq:solve}, and the is also observed in our numerical simulations.

  %However, we also observe in
However, our numerical simulations show that even if we start with initial data that is localized to a small region, it gradually spreads during mixing.
  The incompressibility constraint will of course guarantee that the measure of the support of the solution is conserved in time.
   %It is true that the measure of the support of the solution does not change in time - but
  But since the enstrophy constraint forbids abrupt changes in the velocity field,
  %fluid flow can't vary abruptly in the case of enstrophy constraint,
  the ``region occupied'' by the initial data tends to spread and is likely to eventually ``fill'' the entire torus (see Figure~\ref{fgrSolPlots}).

  A very interesting question is whether there will eventually be a transition in dynamics where
  %a stronger exponential lower bound for the
  the factor in the exponential decay of the mix norm depends only on the volume of the entire domain, and not the size of the region occupied by the initial data.
  Our attempts to get insight into this question numerically were inconclusive, as we ran out of resolution before observing such a regime change.
  What we address here, however, is an interesting link between universality of the exponential lower bound and mixing in domains with boundaries.

  It has been observed formally \cite{bblGouillartKuncioDauchotDubrulleRouxThiffeault,bblGouillartDauchotDubrulleRouxThiffeault} that presence of walls with no slip conditions inhibits mixing; hyperbolic flows which usually lead to exponential decay of various mixing measures lead to only algebraic mixing rates in presence of walls.
  Here, we provide an elementary and rigorous argument showing that universality of the exponential lower bound in mixing with an enstrophy constraint would lead to an algebraic in time lower bound if the initial data is compactly supported away from the boundary and the advecting velocity field vanishes at the boundary.
  Agreement with earlier heuristic arguments is intriguing; but it is not clear to us if one can expect such result to be true in full generality.
  It would be very interesting to know whether the efficient mixing by an incompressible enstrophy constrained flow spreads the initial data over the entire ambient volume and results in the slowdown of the exponential decay.
  We plan to further investigate this issue in the future.

  %However, momentarily suppose the existence of an exponential lower bound with rate \emph{independent} of $\theta_0$.
  %In this case, we provide an elementary scaling argument that shows that a stronger lower bound must hold, and $\norm{\theta(t)}_{H^{-1}}$ decays \emph{algebraically} with time.
  %Our argument is only rigorous if the initial data is compactly supported and the velocity field is required to vanish on the boundary of the unit cube.
  \subsection*{Notational convention, and plan of this paper.}

  \sidenote{Sat 07/20 GI: Revisit this when the paper is actually done.}
  We will assume throughout the paper that $d \geq 2$ is the dimension, and $\T^d$ is the $d$-dimensional torus, with side length $1$.
  All periodic functions are assumed to be $1$-periodic, and we use $m$ to denote the Lebesgue measure on $\T^d$.
  We will use $\norm{f}_{H^s}$ to denote the \emph{homogeneous} Sobolev norms.

  This paper is organized as follows: In Section~\ref{sxnMainTheorem} we describe the notion of $\delta$-mixed data, and prove Theorem~\ref{thmMain}, modulo a few Lemmas.
  In Section~\ref{sxnLemmaProofs} we prove the required lemmas.
  In Section~\ref{sxnNumerics} we present numerics suggesting that the bound stated in Theorem~\ref{thmMain} is not far from optimal, at least for an initial period of time.
  Finally, we conclude this paper with a scaling argument showing that an exponential lower bound on $\norm{\theta}_{H^{-1}}$ with rate independent of $\theta_0$ will imply a stronger \emph{algebraic} lower bound
  for mixing with flows satisfying no-slip boundary condition.

  \section{Rearrangement Costs and the Proof of the Main Theorem.}\label{sxnMainTheorem}%{{{1

  We devote this section to the proof of Theorem~\ref{thmMain}.
  The idea behind the proof is as follows.
  First, if $\norm{\theta}_{H^{-1}}$ is small enough, then its super-level sets are mixed to certain scales (Lemma~\ref{lmaMixNorm} below).
  Second, any flow that starts with an ``unmixed'' set and mixes it to scale $\delta$ has to do a minimum amount of work~\cites{bblBressan03,bblCrippaDeLellis08}.
  Putting these together yields Theorem~\ref{thmMain}.

  We begin by describing the notion of ``mixed to scale $\delta$'', and relate this to the $H^{-1}$ Sobolev norm.
  \begin{definition}\label{dfnSemiMixed} %{{{
    Let $\kappa \in (0, \frac{1}{2})$ be fixed.
    For $\delta > 0$, we say a set $A\subseteq \mathbb{T}^d$ is $\delta$-semi-mixed if
    $$
      \frac{m\paren[\big]{A \cap B(x, \delta) }}{m( B(x,\delta))}\leq 1-\kappa
      \quad\text{for every } x \in \T^d.
    $$
    If additionally $A^c$ is also $\delta$-semi-mixed, then we say $A$ is $\delta$-mixed (or mixed to scale $\delta$).
  \end{definition}%}}}
  \begin{remark}%{{{
    The parameters $\delta$ and $\kappa$ measures the scale and ``accuracy'' respectively.
    The key parameter here is the scale $\delta$, and the accuracy parameter $\kappa \in (0, 1/2)$ only plays an auxiliary role.
    Given a specific initial distribution to mix, $\kappa$ can be chosen to optimize the bound.

    Note that the notion of a set being mixed here is the same as that of Bressan~\cite{bblBressan03}.
    A set being semi-mixed is of course a weaker notion.
    %(For technical reasons we further require $\kappa (1 + \lambda)  < \lambda $.)
    %, where $\lambda \in (0, 1)$ is as in the statement of Theorem~\ref{thmMain}.
    %Thus we will subsequently assume $\kappa < \lambda / (1 + \lambda)$.
  \end{remark}%}}}

  One relation between $\delta$-semi-mixed and negative Sobolev norms is as follows.
  \begin{lemma}\label{lmaMixNorm}%{{{
    Let $\lambda \in (0, 1]$ and $\theta \in L^\infty(\T^d)$.
    Then for any integer $n > 0$, $\kappa \in (0, \frac{\lambda}{1 + \lambda} )$ there exists an explicit constant $c_0 = c_0(d, \kappa, \lambda, n)$ such that
    $$
      \norm{\theta}_{H^{-n}}\leq \frac{\norm{\theta}_{L^\infty} \delta^{n + d/2}}{c_0}
      \implies
      A_\lambda \text{ is $\delta$-semi-mixed.}
    $$
    Here $A_\lambda$ is the super level set defined by
      $A_\lambda \defeq \{\theta > \lambda \norm{\theta}_{L^\infty} \}$.
  \end{lemma}%}}}

  Our interest in this Lemma is mainly when $n = 1$.
  Note that while Lemma~\ref{lmaMixNorm} guarantees the super level sets $A_\lambda$ are $\delta$-semi-mixed, they need not be $\delta$-mixed.
  Indeed if $A_\lambda$ is very small, its complement won't be $\delta$-semi-mixed.
  Also, we remark that the converse of Lemma~\ref{lmaMixNorm} need not be true.
  For example the function
  $$
    f(x) = \sin(2 \pi x) + 10 \sin( 2\pi n x )
  $$
  has $\norm{f}_{H^{-1}(\T^1)} = O(1)$, and the super level set $\{ f > 5 \}$ is certainly semi-mixed to scale $1/n$ (see also~\cite{bblLinThiffeaultDoering11}).

  The proof of Lemma~\ref{lmaMixNorm} follows from a duality and scaling argument.
  For clarity of presentation we postpone the proof to Section~\ref{sxnLemmaProofs}.
  Returning to Theorem~\ref{thmMain}, the main ingredient in its proof is a lower bound on the ``amount of work'' required to mix a set to fine scales.
  This notion goes back to a conjecture of Bressan for which a \$500 prize was announced~\cite{bblBressanPrize}.
  \begin{conjecture}[Bressan '03~\cite{bblBressan03}]\label{cjrBressan}
    Let $H$ to be the left half of the torus, and $\Psi$ be the flow generated by
    an incompressible
    %a nearly incompressible
    vector field $u$.
    If after time $T$ the image of $H$ under the flow $\Psi$ is $\delta$-mixed, then there exists a constant $C$ such that
    \begin{equation}\label{eqnBressan}
      \int_0^T \norm{\grad u(\cdot, t)}_{L^1} \, dt \geq \frac{ \abs{\ln \delta } }{C} .
    \end{equation}
  \end{conjecture}

  We refer the reader to~\cite{bblBressan03} for the motivation of the lower bound~\eqref{eqnBressan} and further discussion.
  To the best of our knowledge, this conjecture is still open.
  %and the \$500 prize has not been claimed.
  However, Crippa and De Lellis~\cite{bblCrippaDeLellis08} made significant progress towards the resolution of this conjecture.
  \begin{theorem}[Crippa, De Lellis '08~\cite{bblCrippaDeLellis08}]\label{thmCrippaDeLellis}
    Using the same notation as in Conjecture~\ref{cjrBressan}, for any $p > 1$ there exists a finite positive constant $C_p$ such that
    \begin{equation}\label{eqnCrippaDeLellis}
      \int_0^T \norm{\grad u(\cdot, t)}_{L^p} \, dt \geq \frac{\abs{\ln \delta}}{C_p}.
    \end{equation}
  \end{theorem}

  For our purposes we will need two extensions of Theorem~\ref{thmCrippaDeLellis}.
  First, we will need to start with sets other than the half torus.
  Second, we will need lower bounds for the work done to \emph{semi}-mix sets to small scales.
  Note that in order for a flow to $\delta$-mix a set $A$, it has to both $\delta$-semi-mix $A$ and $\delta$-semi-mix $A^c$.
  Generically each of these steps should cost comparable amounts, and hence a semi-mixed version of Theorem~\ref{thmCrippaDeLellis} should follow using techniques in~\cite{bblCrippaDeLellis08}.
  We state this as our next lemma.
  \begin{lemma}\label{lmaSemiMixedWorkDone}%{{{
    Let $\Psi$ be the flow map of an incompressible vector field $u$.
    Let $A \subset \T^d$ be any measurable set and let $p > 1$.
    There exist constants $r_0 = r_0(A)$ and $a = a(d, \kappa, p) > 0$, such that if for some $\delta < r_0 / 2$ and $T > 0$ the set $\Psi_T(A)$ is $\delta$-\emph{semi}-mixed, then
    \begin{equation}\label{eqnSemiMixed}
      \int_0^T \norm{\grad u(\cdot, t)}_{L^p} \, dt \geq \frac{ m(A)^{1/p} }{a} \abs[\Big]{\ln \frac{2 \delta}{r_0} }.
    \end{equation}
  \end{lemma}%}}}

  Morally the constant $r_0$ above should be a length scale at which set $A$ is not semi-mixed.
  Our proof, however, uses a condition on $r_0$ which is slightly stronger than only requiring that $A$ is not semi-mixed to scale $r_0$.
  Namely, we will require ``most'' of $A$ to occupy ``most'' of the union of disjoint balls of radius at least $r_0$.
  %Our assumption that $A$ contains a ball of radius $r_0$ is to ensure that the set $A$ is not semi-mixed to scale $r_0$.
  %This assumption can in fact be relaxed to assuming that there exists $\kappa' < \kappa$, such that $A$ is not $r_0$ semi-mixed (with constant $\kappa'$), provided we define the notion of $\delta$-semi-mixed using cubes, instead of balls.
  %The reason we need cube is that the proof of this extension uses a counting argument which requires us to tile a region with length scale $r_0$ with regions of length scale $\delta$.
  %Cubes, of course, stack nicely; but balls do not.
  %We don't present this extension here as it only obscures the heart of the matter, and the version stated above will suffice for our purposes.
  Deferring the proof of Lemma~\ref{lmaSemiMixedWorkDone} to Section~\ref{sxnLemmaProofs}, we prove Theorem~\ref{thmMain}.

  \begin{proof}[Proof of Theorem~\ref{thmMain}]%{{{
    Replacing $\theta$ with $\theta / \norm{\theta}_{L^\infty}$, we may without loss of generality assume $\norm{\theta_0}_{L^\infty} = 1$.
    %Replacing $\theta$ with $-\theta$ if necessary, we may further assume set $A_0 \defeq \{ \theta_0 > \lambda \}$ contains a ball of radius $r_0$.
    %Choose $x_0 \in A_0$ so that $\theta_0(x_0) = 1$.
    %By the mean value theorem we immediately see that
    %$$
    %  B\paren[\big]{x_0, \frac{1}{3 \norm{\grad \theta_0}_{L^\infty}} } \subset A_0,
    %$$
    %so Lemma~\ref{lmaSemiMixedWorkDone} can be applied to the set $A_0$.
    Fix $0< \lambda \leq 1,$ and
     $\kappa \in (0, \frac{\lambda}{1 + \lambda})$. Let $a$ be the constant from Lemma~\ref{lmaSemiMixedWorkDone}, and $c_0$ the constant from Lemma~\ref{lmaMixNorm} with $n = 1$.
    Choose
    $$
      \delta = \paren[\Big]{ c_0 \norm{\theta(t)}_{H^{-1} } }^{\frac{2}{d+2}}.
    $$
    Then certainly $\norm{\theta(t)}_{H^{-1}} \leq \delta^{d/2+1}/ c_0$ and by Lemma~\ref{lmaMixNorm} the super level set $\{ \theta(t) > \lambda \}$ is $\delta$-semi-mixed.

    Now, since $\theta$ satisfies the transport equation~\eqref{eq:1}, we know
    $
      \set{ \theta(t) > \lambda } = \Psi_t(A_\lambda)
    $,
    where $\Psi$ is the flow of the vector field $u$.
    Thus, Lemma~\ref{lmaSemiMixedWorkDone} now implies
    $$
      \delta \geq \frac{r_0}{2} \exp \paren[\Big]{ \frac{-a}{m(A_\lambda)^{1/p} } \int_0^t \norm{\grad u}_{L^p} }.
    $$
    Consequently
    $$
      \norm{\theta(t)}_{H^{-1}}
	= \frac{\delta^{d/2+1} }{c_0}
	\geq
	  { \frac{r_0^{d/2+1}}{c_0 2^{d/2+1}} } \exp \paren[\Big]{ \frac{-d a}{m(A_\lambda)^{1/p} } \int_0^t \norm{\grad u}_{L^p} },
    $$
    finishing the proof.
  \end{proof}%}}}

  \section{Proofs of Lemmas.}\label{sxnLemmaProofs}%{{{1

  We devote this section to the proofs of Lemmas~\ref{lmaMixNorm} and~\ref{lmaSemiMixedWorkDone}.
  %The proof of Lemma~\ref{lmaMixNorm} is quick, and we present it first.

  \begin{proof}[Proof of Lemma~\ref{lmaMixNorm}]%{{{
    Suppose for the sake of contradiction that $A_\lambda$ is not $\delta$-semi-mixed.
    Then by definition, there exists $x \in \T^d$ such that
    \begin{equation}\label{eqnDual1}
      m(A_\lambda \cap B(x,\delta))
	\geq (1-\kappa) m(B(x,\delta))
	=(1-\kappa)\pi(d) \delta^d.
    \end{equation}
    Here $\pi(d)$ is the volume of $d$-dimensional unit ball.

    By duality
    \begin{equation}\label{eqnDual2}
      \norm{\theta}_{H^{-n}}
	= \sup_{f\in H^{n}} \frac{1}{\norm{f}_{H^{n}} } \abs[\Big]{\int_{\mathbb{T}^d} \theta(x)f(x) \, dx }.
    \end{equation}
    We choose $f \in H^{n}$ to be a function which is identically equal to $1$ in $B(x,\delta)$, and which vanishes outside $B(x,(1+\eps)\delta)$ for some small $\eps > 0$.
    A direct calculation shows that we can arrange
    $$
      \norm{f}_{H^{n}}\leq c_1(d)\cdot \eps^{-n + \frac{1}{2}}\cdot \delta^{-n + \frac{d}{2}},
    $$
    for some (explicit) constant $c_1$ depending only on the dimension.
    %We subsequently assume that $c_1$ can change from line to line, provided it only depends on the dimension $d$.

    On the other hand using~\eqref{eqnDual1} gives
    \begin{equation}\label{eq:kap}
      \int_{\mathbb{T}^d}\theta(x) f(x) dx
	\geq \pi(d) \norm{\theta}_{L^\infty} \delta^d \paren[\big]{ (1-\kappa) \lambda - \kappa - c_2(d)\eps},
    \end{equation}
    for some (explicit) dimensional constant $c_2(d)$.
    Choosing
    $
      \eps = \frac{\lambda - (1+\lambda)\kappa}{2c_2(d)}
    $
    and using~\eqref{eqnDual2} we obtain
    %And since our choice of $f$ doesn't depend on the location of the ball, so if we take $\kappa=\frac{1}{5}$, the right hand side of (\ref{eq:kap}) is no less than $C(d)\delta^d$, with $C(d)$ only depend on dimension.
    %So in all, we have
    $$
      \norm{\theta}_{H^{-n}}\geq \frac{\norm{\theta}_{L^\infty} \delta^{-n + \frac{d}{2}} }{c_0(d, \kappa, \lambda, n)}
    $$
    as desired.
  \end{proof}%}}}
  \begin{remark*}%{{{
    Observe $c_0 = c_0'(d, n) (\lambda - (1 + \lambda) \kappa )^{n - \frac{1}{2}}$. % where $c_0'$ is a constant that only depends on $d$ and $n$.
  \end{remark*}%}}}

  Now we turn to Lemma~\ref{lmaSemiMixedWorkDone}.
  For this, we need a result from~\cite{bblCrippaDeLellis08} which controls the Lipshitz constant of the Lagrangian map except on a set of small measure.
  \begin{proposition}[Crippa DeLellis '08~\cite{bblCrippaDeLellis08}]\label{ppnLipConstant}%{{{
    Let $\Psi(t,x)$ be the flow map of the (incompressible) vector field $u$.
    For every $p > 1$, $\eta>0$, there exists a set $E \subset \T^d$ and a constant $c = c(d, p)$ such that $m(E^c)\leq \eta$ and for any $t \geq 0$
    we have
    \begin{equation}\label{eqnLipConstant}
      \Lip(\Psi^{-1}(t,\cdot)|_{E^c})
	\leq \exp \paren[\Big]{ \frac{c}{\eta^{\frac{1}{p}} } \int_0^t \norm{\grad u(s)}_{L^p} \, ds }.
    \end{equation}
    Here
    $$
      \Lip(\Psi^{-1}(t,\cdot)|_{E^c})
	  \defeq \sup_{\substack{x, y \in E^c\\x \neq y}} \frac{\abs{\Psi\inv(t, x) - \Psi\inv(t, y)}}{\abs{x - y}}
    $$
    is the Lipshitz constant of $\Psi\inv$ on $E^c$.
  %Here $M=\norm{\nabla u}_{L^1([0,t],L^2)}$.
  \end{proposition}%}}}

  The proof of Proposition~\ref{ppnLipConstant} is built
  %, at its root,
  upon the simple observation~\cite{bblAmbrosioLecumberryStefania05} that for a passive scalar $\theta(x,t)$ and smooth advecting velocity $u$ one has the inequality
  \begin{equation}\label{logineq11}
    \int \log_+ \abs[\big]{ \nabla \theta\paren[\big]{ t, \Psi(t,x)} } \,dx
      \leq \int_0^t \int \abs[\big]{ \nabla u\paren[\big]{ t, \Psi(t,x) } }\,dx.
  \end{equation}
  This can be proved by an elementary calculation.
  In fact, even the point wise bound
  \[ D \log |\nabla \theta| \leq |\nabla u| \]
  is true, where $D = \delt + u \cdot \grad$ is the material derivative.
  In the form \eqref{logineq11}, this inequality is not very useful.
  But it turns out that the more sophisticated maximal form of this inequality \cite{bblAmbrosioLecumberryStefania05,bblCrippaDeLellis08} can be much more useful and is essentially what leads to Proposition~\ref{ppnLipConstant}.
  We refer the reader to~\cite{bblCrippaDeLellis08} for the details of the proof.

  We use Proposition~\ref{ppnLipConstant} to prove Lemma~\ref{lmaSemiMixedWorkDone} below.

  \begin{proof}[Proof of Lemma~\ref{lmaSemiMixedWorkDone}.]%{{{
    %By assumption $A$ contains a ball of radius $r_0$.
    %Let $\bar x$ be the center of this ball.
    The main idea behind the proof is as follows:
    Suppose first $r_0$ is some large scale at which the set $A$ is ``not semi-mixed''.
    Let $T > 0$ be fixed and suppose $\Psi_T(A)$ is $\delta$-semi-mixed for some $\delta < r_0 / 2$.
    Since $\Psi_T(A)$ is $\delta$-semi-mixed, there should be many points $\tilde x \in \Psi_T(A)$ and $\tilde y \in \Psi_T(A)^c$ such that $\abs{\tilde x - \tilde y} < \delta$.
    Since $A$ is ``not semi-mixed'' to scale $r_0$, there should be many points $\tilde x$ and $\tilde y$ so that we additionally have $\abs{\Psi_T\inv(\tilde x) - \Psi_T\inv(\tilde y)} \geq r_0 / 2$.
    This will force the Lipshitz constant of $\Psi_T\inv$ to be at least $r_0 / (2 \delta)$ on a set of large measure.
    Combined with Proposition~\ref{ppnLipConstant} this will give the desired lower bound on $\int_0^t \norm{\grad u}_{L^p}$.

    We now carry out the details of the above outline.
    The first step in the proof is to choose the length scale $r_0$.
    Let $\eps = \eps( \kappa, d )$ be a small constant to be chosen later.
    We claim that there exists a natural number $l$ and finitely many disjoint balls $B(x_1, r_1)$, \dots, $B(x_l, r_l)$ such that
    \begin{equation}\label{eqnDisjBalls}
      %For every $i \in \{1, \dots, M\}$ we have
      m\paren[\Big]{ \bigcup_{i=1}^l B(x_i, r_i) } \geq \frac{m(A)}{2 \cdot 3^d}
      \quad\text{and}\quad
      \frac{ m( A \cap B(x_j, r_j) ) }{m(B(x_j, r_j))} > 1 - \eps
    \end{equation}
    for every $j \in \{1, \dots l\}$.

    To see this, note that the metric density of $A$ is $1$ almost surely in $A$.
    Thus, removing a set of measure $0$ from $A$ if necessary, we know that for every $x \in A$ there exists an $r \in (0, 1]$ such that
    $$
      \frac{ m(A \cap B(x, r) ) }{m(B(x, r))} > 1 - \eps.
    $$
    Now choose $K \subset A$ compact with $m(K) > m(A) / 2$.
    Since the above collection of balls is certainly a cover of $K$, we pass to a finite sub-cover.
    Applying Vitali's lemma to this sub-cover we obtain a disjoint sub-family $\{B(x_i, r_i) \mid i = 1, \dots, l \}$ with $m(\cup B(x_i, r_i) ) \geq m(K) / 3^d$.
    This immediately implies~\eqref{eqnDisjBalls}.
    For convenience let $B_i = B(x_i, r_i)$, and choose $r_0 = \min\{r_1, \dots, r_l\}$.

    Now let $\eta > 0$ be another small parameter that will be chosen later.
    By Proposition~\ref{ppnLipConstant} we know that there exists a set $E$ with $m(E)\leq\eta$ such that the inequality~\eqref{eqnLipConstant} holds.
    Define the set
    \begin{equation}
      %\tilde{A}=\{x\in A: B(\Psi_T(x),\delta)\cap [\Psi_T(A^c) - E]=\emptyset\}.
      F = \set[\big]{ x \in \T^d \;\big|\;  \frac{ m( B( x ,\delta) \cap E)}{m( B(x,\delta))} > \frac{\kappa}{2} }
    \end{equation}
    Clearly $F \subset \{ M \Chi*E > \kappa / 2 \}$, where $M \Chi*E$ is the maximal function of $\Chi*E$.
    Consequently,
    $$
      m( F )
	\leq m\paren[\big]{  \{ M \Chi*E > \frac{\kappa }{2} \} }
	%\leq c m(E),
	\leq \frac{2 c_1}{\kappa} m(E)
	%\leq \frac{2 c_d}{\kappa} \eta
    $$
    for some explicit constant $c_1 = c_1(d)$.
    (It is well known that $c_1 = 3^d$ will suffice.)
    %We will subsequently allow $c$ to change from line to line, provided it only depends on $d$ and $\kappa$.

    Since $\Psi_T$ is measure preserving we know $m(\Psi_T\inv (E \cup F)) \leq (1 + 2 c_1 / \kappa) \eta$.
    Thus choosing
    $$
      \eta = \frac{\kappa}{\kappa + 2 c_1} \paren[\Big]{ \frac{1}{4^d} - \eps } \sum_{i = 1}^l m\paren[\big]{ B_i }
    $$
    will guarantee
    $$
      m( \Psi_T\inv( E \cup F) )
	\leq \paren[\Big]{ \frac{1}{4^d} - \eps } \sum_{i = 1}^l m\paren[\big]{ B_i }.
    $$
    This implies that for some $i_0 \leq l$ we must have
    \begin{equation}\label{eqnCovering1}
      m( (B_{i_0}  \cap A) - \Psi_T\inv(E \cup F) )
	\geq \paren[\Big]{ 1 - \frac{1}{4^d} } m\paren[\big]{ B_{i_0} }.
    \end{equation}
%    for some $i_0 \leq l$.
    By reordering, we may without loss of generality assume that $i_0 = 1$.
    Consequently, for
    $$
      C_1 = \set[\Big]{ x \in (B_1 \cap A) - \Psi_T\inv( E \cup F) \;\Big|\; d( x, B_1^c ) > \frac{r_1}{2}}.
    $$
    equation~\eqref{eqnCovering1} implies
    $$
      m( C_1 ) \geq \paren[\Big]{ \frac{1}{2^d} - \frac{1}{4^d} } m( B_1 ).
    $$

    Now, from the collection of open balls $\{ B( \tilde x, \delta ) \mid \tilde x \in \Psi_T(C_1) \}$ the Vitali covering lemma allows us to extract a finite disjoint collection $B( \tilde x_1, \delta )$, \dots, $B( \tilde x_n, \delta )$ such that
    $$
      m\paren[\Big]{  \bigcup_1^n B( \tilde x_i, \delta ) } \geq \frac{m( C_1 )}{5^d}.
    $$
    Our goal is to find $\tilde y$ such that $\tilde y \in B( \tilde x_i, \delta ) - E$ for some $i$, and $\abs{\Psi_T\inv \tilde y - \Psi_T\inv x} > r_1 / 2$.

    For convenience set $\tilde B_i = B( \tilde x_i, \delta )$.
    Since $\Psi_T(A)$ is $\delta$-semi-mixed and $\tilde x_i \not\in F$ we have
    \begin{equation}\label{eqnCovering2}
      m( \Psi_T(A) \cap \tilde B_i ) \leq (1-\kappa) m( \tilde B_i )
      \quad\text{and}\quad
      m( E \cap \tilde B_i ) \leq \frac{\kappa}{2} m( \tilde B_i ).
    \end{equation}
    Also, since $\Psi_T$ is measure preserving and by the definition of $B_1$ we see
    \begin{equation}\label{eqnCovering3}
      m\paren[\Big]{ \bigcup_{i=1}^n \tilde B_i \cap \Psi_T\paren[\big]{ B_1 - A } }
	\leq m( B_1 - A)
	< \eps m(B_1)
	%< \eps \paren[\Big]{ \frac{1}{2^d} - \frac{1}{4^d} } m(C_1)
    \end{equation}

    Using the fact that $\{ \tilde B_i \}$ are all disjoint, summing~\eqref{eqnCovering2} and using~\eqref{eqnCovering3} gives
    \begin{multline*}
      m( \bigcup_{i=1}^n \tilde B_i \cap E \cap \Psi_T (B_1) )
	< \paren[\big]{ 1 - \frac{\kappa}{2} } \sum_{i = 1}^n m( \tilde B_i ) + \eps m(B_1)\\
	\leq
	  \paren[\Big]{ 1 - \frac{\kappa}{2} + \eps 5^d \paren[\Big]{ \frac{1}{2^d} - \frac{1}{4^d} }\inv } \sum_{i = 1}^n m( \tilde B_i ).
    \end{multline*}
    Thus choosing
    $$
      \eps < \frac{\kappa}{2 \cdot 5^d} \paren[\Big]{ \frac{1}{2^d} - \frac{1}{4^d} }
    $$
    will guarantee
    $$
      m\paren[\big]{ \bigcup_{i=1}^n \tilde B_i \cap E \cap \Psi_T (B_1) }
      < m\paren[\big]{ \bigcup_{i=1}^n \tilde B_i }
    $$
    This in turn will guarantee that for some $i$ we can find $\tilde y \in \tilde B_{i} - E - \Psi_T(B_1)$.

    Now observe that
    $$
      \tilde y, \tilde x_i \not\in E,
      \qquad
      \abs{\tilde y - \tilde x_i} < \delta,
      \qquad\text{and}\qquad
      \abs{\Psi_T\inv( \tilde y ) - \Psi_T\inv (\tilde x_i ) } > \frac{r_1}{2}.
    $$
    The last inequality above follows because $\Psi_T\inv(\tilde x_i) \in C_1$ and $\Psi_T\inv( \tilde y ) \not\in B_1$.
    This forces
    $$
      \Lip( \Psi_T\inv |_{E^c} )
	\geq \frac{ \abs{\Psi_T\inv(\tilde y) - \Psi_T\inv(\tilde x_i)} }{\abs{\tilde y - \tilde x_i}}
	> \frac{r_1}{2 \delta} \geq \frac{r_0}{2 \delta}.
    $$
    Now using~\eqref{eqnLipConstant}, and letting $a = a(d, \kappa, p)$ denote a constant that changes from line to line we obtain
    \begin{equation}\label{eqnLipBd1}
      \int_0^T \norm{\grad u(t)}_{L^p} \, dt
	\geq \frac{\eta^{\frac{1}{p}} }{a} \abs[\big]{ \log \paren[\big]{\frac{r_0}{2\delta}} }.
	%= \frac{r_0^{d/2}}{a} \abs[\big]{ \log \paren[\big]{\frac{2\delta}{r_0}}  },
    \end{equation}
    Observe finally that
    $$
      \eta =
	c_2 m\paren[\big]{ \bigcup_{i =1}^l B_i }
	\geq \frac{c_2 m(A)}{2 \cdot 3^d}
	%\leq \frac{c_2}{1 - \eps} m(A)
    $$
    for some explicit constant $c_2 = c_2( d, \kappa )$.
    Consequently~\eqref{eqnLipBd1} reduces to
    $$
      \int_0^T \norm{\grad u(t)}_{L^p} \, dt
	\geq \frac{m(A)^{\frac{1}{p}} }{a} \abs[\big]{ \log \paren[\big]{\frac{r_0}{2\delta}} },
    $$
    as desired.
  \end{proof}%}}}

  \section{Numerical results}\label{sxnNumerics}%{{{1

  In this section we present numerical results illustrating how the exponential decay rate varies with the initial data.
  For numerical purposes we work on the $1$-periodic torus.
  Given a parameter $a$, we define the initial data $\theta_0 = \theta_0' / \norm{\theta_0'}_{L^2}$ where
  $$
    \theta_0'(x, y) = \begin{dcases}
      \sin\paren[\big]{ \frac{2 \pi x}{a} } \sin\paren[\big]{ \frac{2 \pi (y + \frac{a}{8}) }{a} }
	& \text{for } 0 < x < \frac{a}{2} ~\text{and}~ \frac{-a}{8} < y < \frac{a}{2} - \frac{a}{8}
      \\
      \sin\paren[\big]{ \frac{2 \pi x}{a} } \sin\paren[\big]{ \frac{2 \pi (y - \frac{a}{8}) }{a} }
	& \text{for } \frac{a}{2} < x < a ~\text{and}~ \frac{a}{8} < y < \frac{a}{2} + \frac{a}{8}\\
	0 & \text{otherwise}.
    \end{dcases}
    %\theta_0'(x, y) = \Chi*B \paren[\big]{ \sin(mx)\sin(my)+10\sin(4mx)\sin(4my) },
  $$
  A figure of this is shown in~\ref{fgrSolPlots}(a).

We do not know which velocity field achieves the lower bound~\eqref{eq:solve}. However the steepest descent method introduced in~\cite{bblLinThiffeaultDoering11} provides us with a reasonable candidate.
  Explicitly, their formula gives
  \begin{equation}\label{eqnVelocity}
      u=\frac{-\lap^{-1}P(\theta\nabla^{-1}\theta)}{\norm{\nabla^{-1}P(\theta\nabla^{-1}\theta)}_{L^2}},
  \end{equation}
  where $P$ is the Leray-Hodge projection onto divergence free vector fields.
  This can be derived by multiplying both sides of~\eqref{eq:1} by ${\lap}^{-1}\theta$ and integrating by parts.

  Using a pseudo-spectral method%
  \footnote{The code and more figures can be downloaded from~\cite{bblWebsite}.}
  retaining $768$ Fourier modes in each variable we perform a numerical simulation of~\eqref{eq:1} with the initial data obtained by varying the parameter $a$ over the set $\{ 6/12, 7/12, \dots, 11/12\}$, and the velocity obtained dynamically using~\eqref{eqnVelocity}.
  Plots of our solutions at various times (for $a = 11/12$) are shown in Figure~\ref{fgrSolPlots}.
  \begin{figure}[htb]%{{{
    \centering
    \subfigure[$t=0$]{
      \includegraphics[width=.27\linewidth]{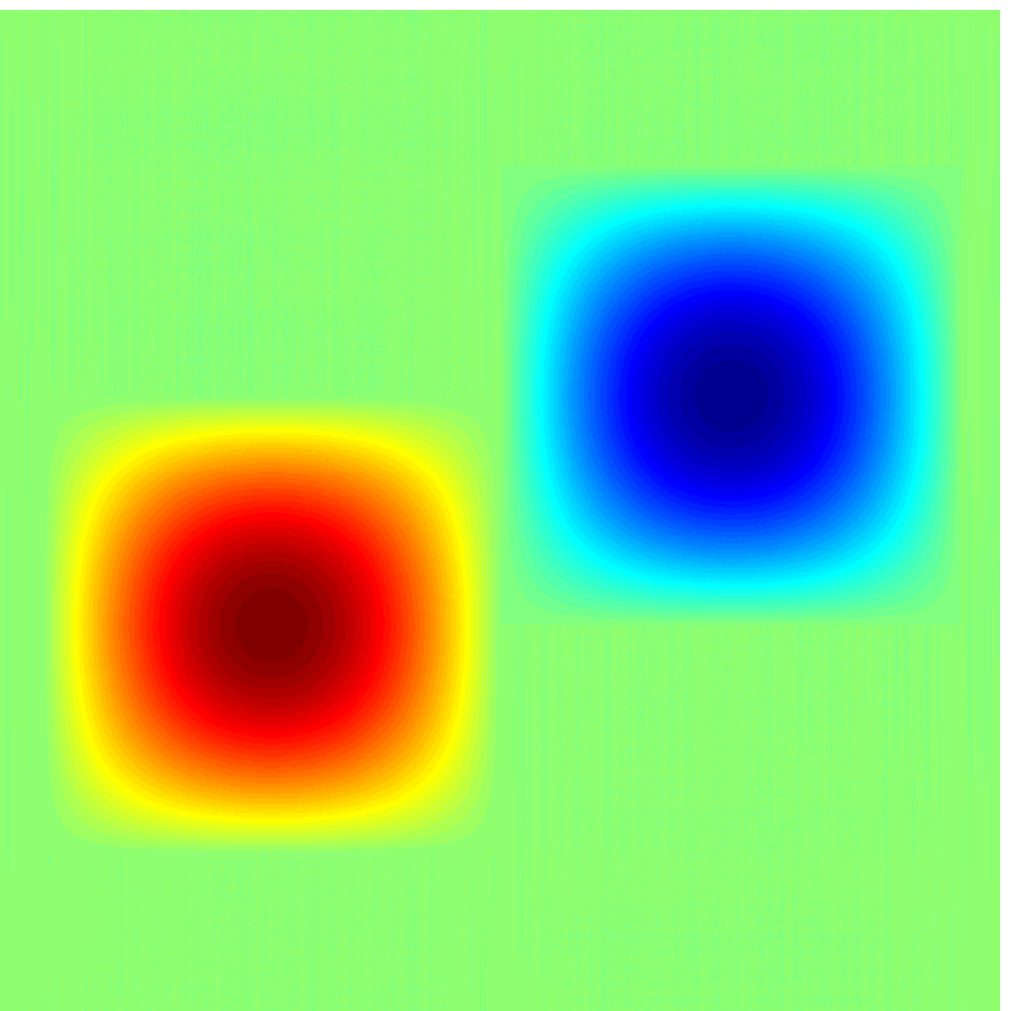}
    }\quad
    \subfigure[$t=1$]{
      \includegraphics[width=.27\linewidth]{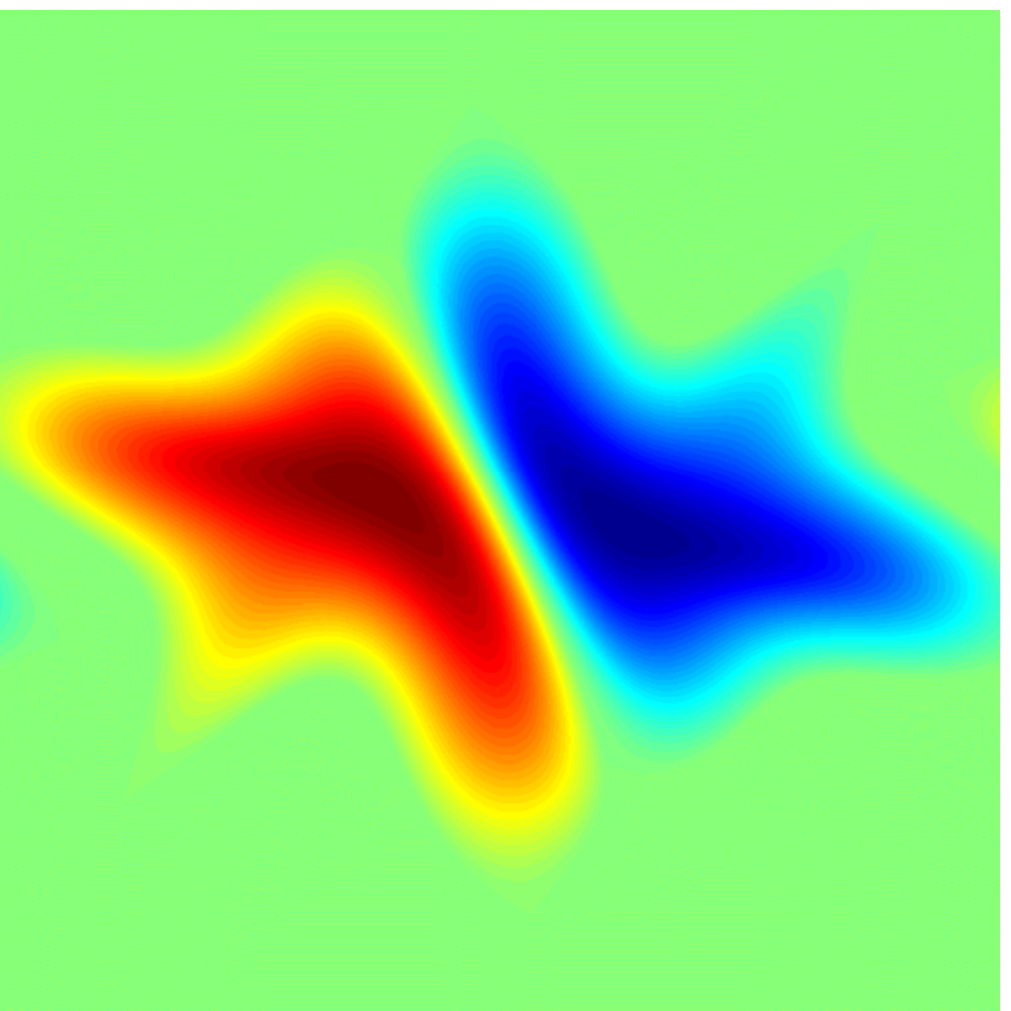}
    }\quad
    \subfigure[$t=2.05$]{
      \includegraphics[width=.27\linewidth]{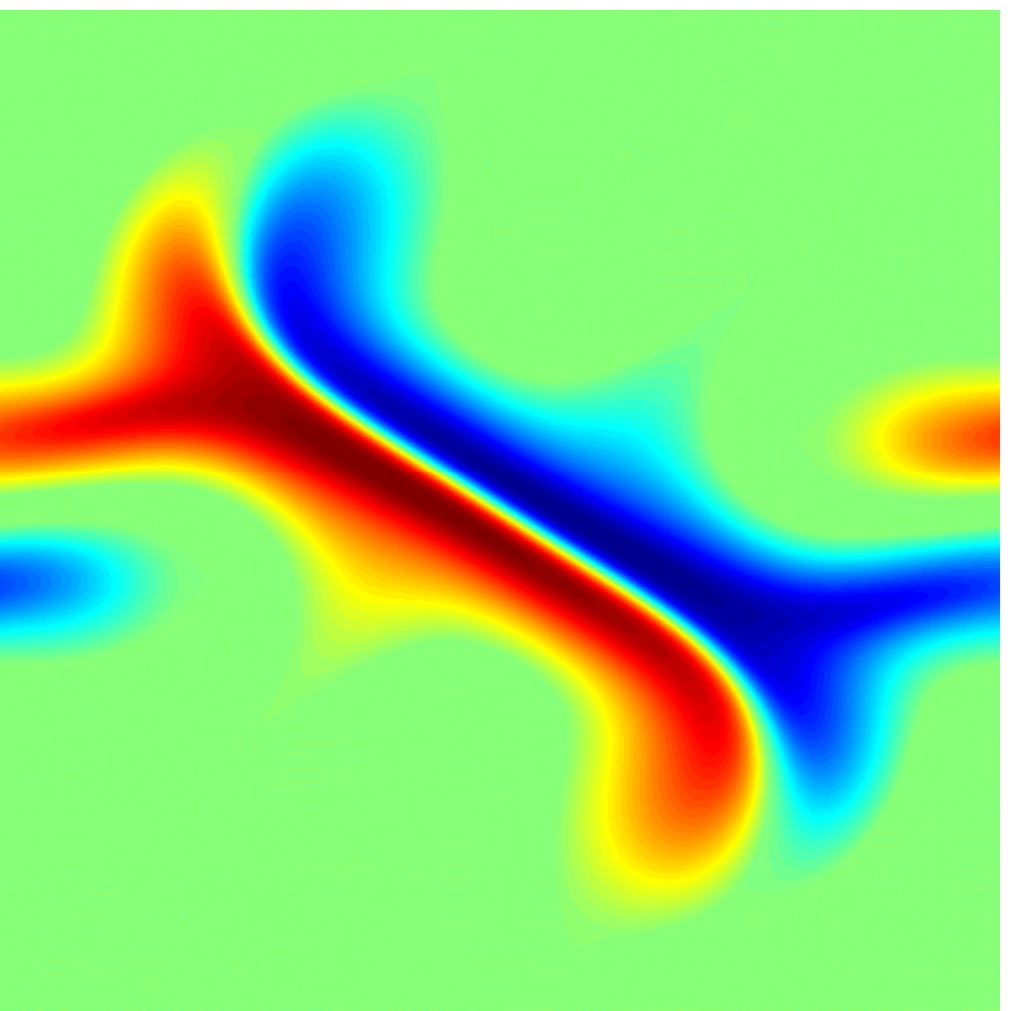}
    }\\
    \subfigure[$t=3.1$]{
      \includegraphics[width=.27\linewidth]{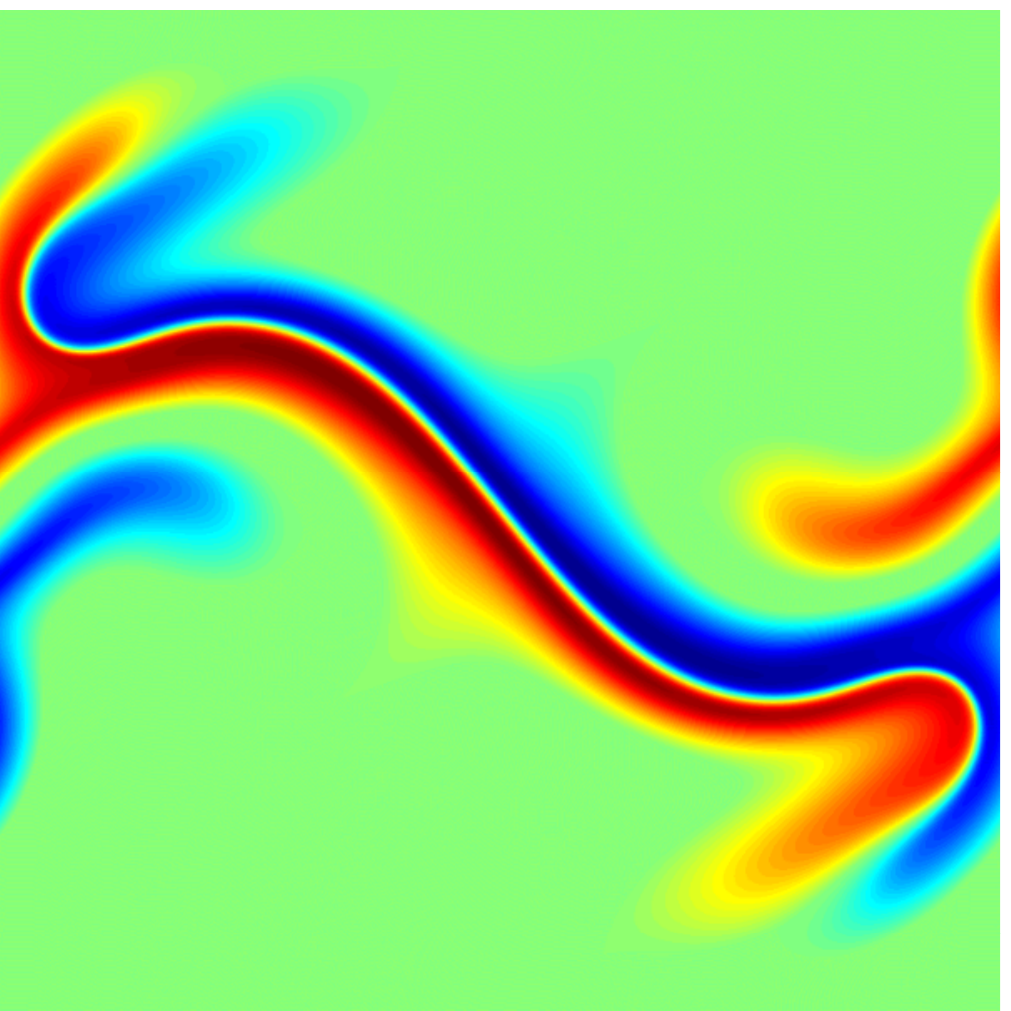}
    }\quad
    \subfigure[$t=4.15$]{
      \includegraphics[width=.27\linewidth]{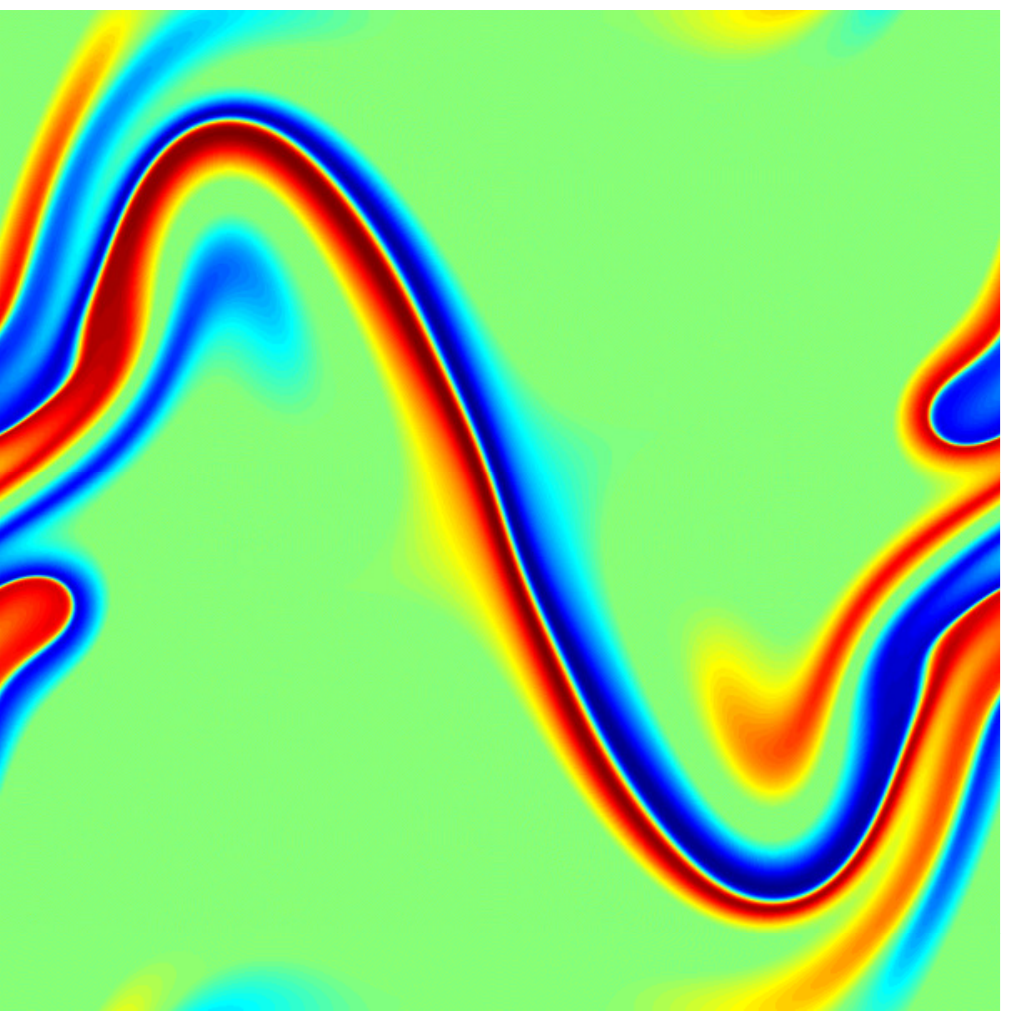}
    }\quad
    \subfigure[$t=5.19$]{
      \includegraphics[width=.27\linewidth]{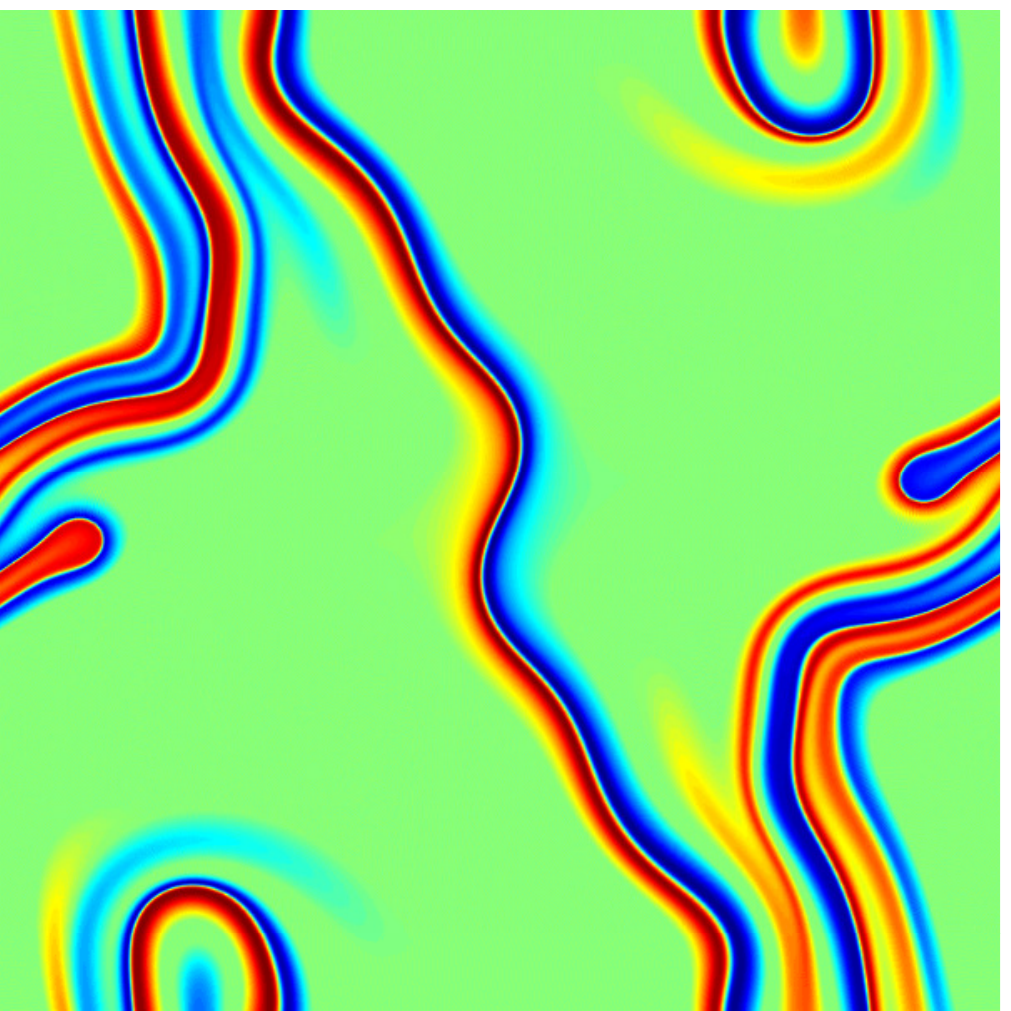}
    }
    \caption{Solution plots at various times for $a=11/12$.}
    \label{fgrSolPlots}
  \end{figure}%}}}

  Figure~\ref{fgrNumerics}(a) shows graphs of $\norm{\theta(t)}_{H^{-1}}$ vs $t$ as the parameter $a$ varies over the set $\{6/12, \dots, 11/12\}$.
  Figure~\ref{fgrNumerics}(b) shows graphs of $\ln \norm{\theta(t)}_{H^{-1}}$ vs $t$ for the same values of $a$.
  Following a short initial ``settling down'' period, the log plots in Figure~\ref{fgrNumerics}(b) are essentially linear indicating a exponential in time decay of $\norm{\theta_0}_{H^{-1}}$.
  \begin{figure}[htb]
    \centering
    \subfigure[$\norm{\theta(t)}_{H^{-1}}$ vs $t$]{
      \includegraphics[width=.27\linewidth]{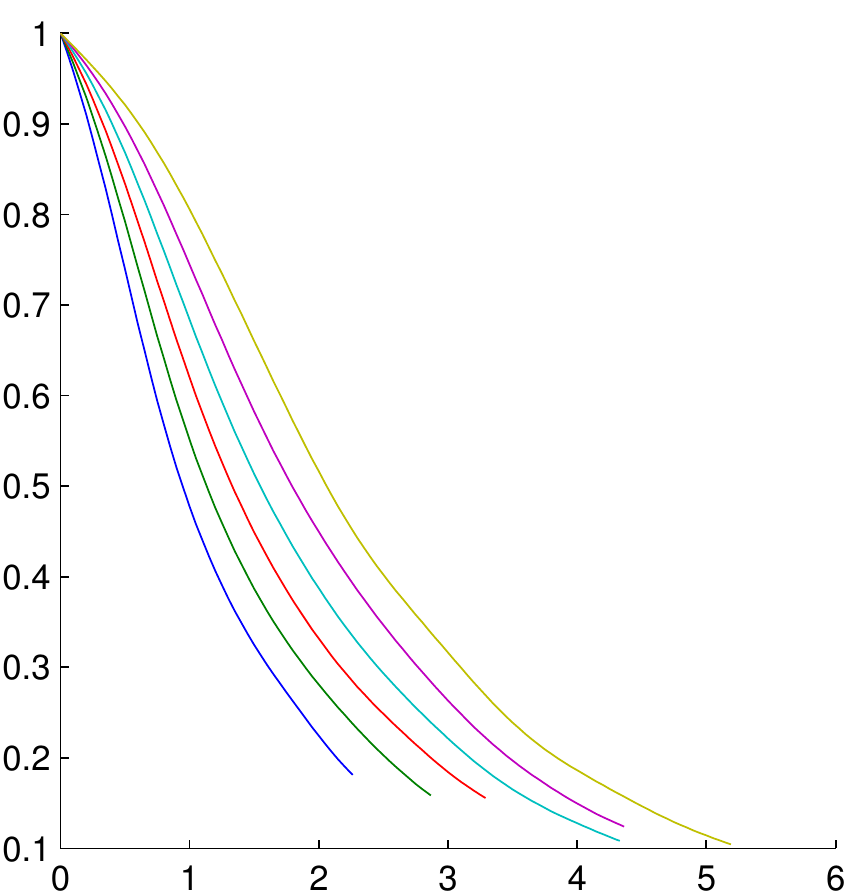}
    }\quad
    \subfigure[$\ln \norm{\theta(t)}_{H^{-1}}$ vs $t$]{
      \includegraphics[width=.27\linewidth]{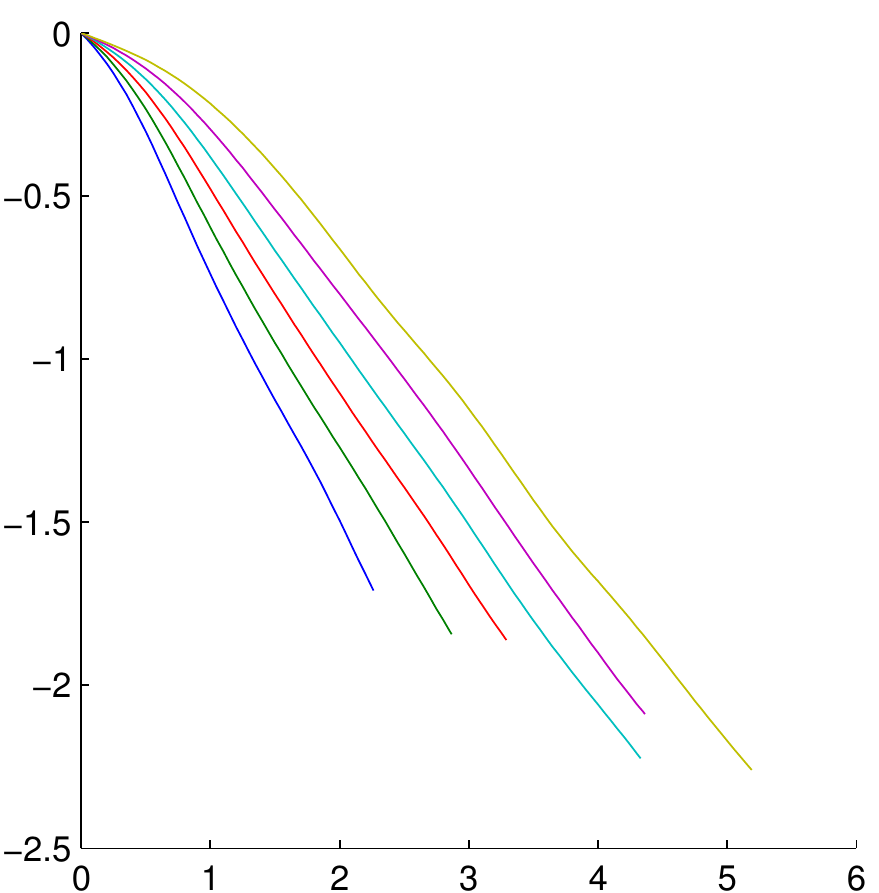}
    }\quad
    \subfigure[Exponential decay rate.]{
      \includegraphics[width=.27\linewidth]{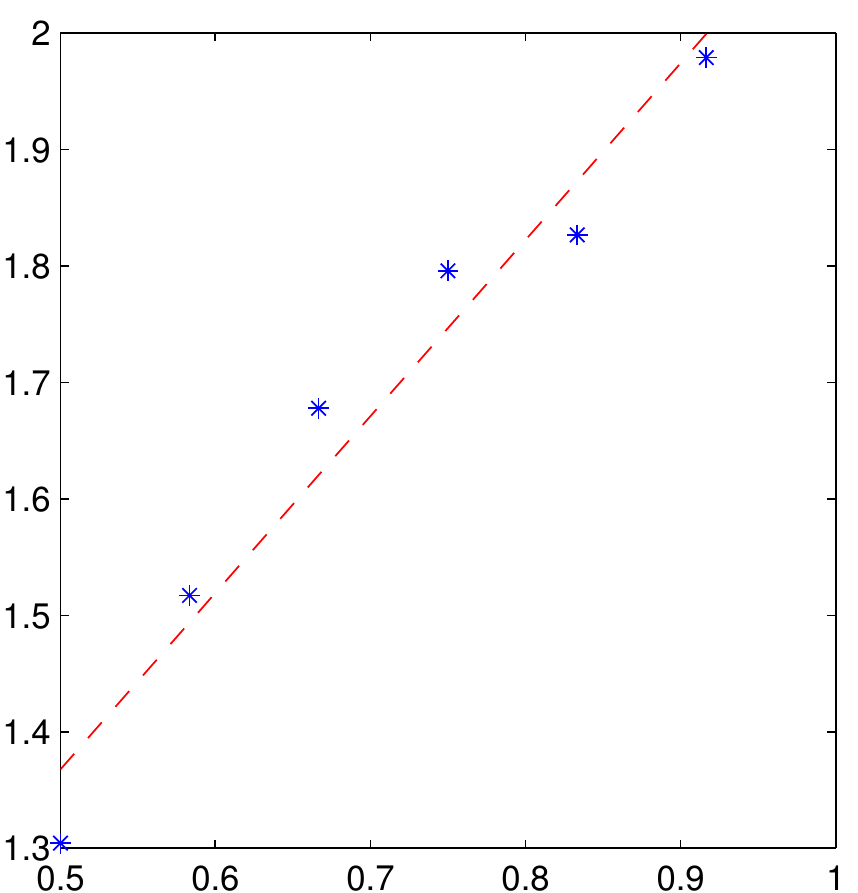}
    }
    \caption{The mix norm of the scalar density (Figures (a) \& (b)), and the negative reciprocal of the exponential decay rate vs $a$ as $a$ varies over $\{6/12, \dots, 11/12\}$ (Figure (c)).}
    \label{fgrNumerics}
  \end{figure}

  We fit each of the log plots in Figure~\ref{fgrNumerics}(b) to a straight line, and plot the negative reciprocal of the slope vs $a$ in Figure~\ref{fgrNumerics}(c).
  Since $m( \supp(\theta_0) ) = O( a^2 )$, Theorem~\ref{thmMain} predicts this graph to be linear as a function of $a$.
  %These appear to grow linearly with $m$.
  %Note that decay rate predicted by~\eqref{eq:solve} in two dimensions is $O(1/r_0)$, where $r_0$ is the radius of the largest ball contained in a super level set of $\theta_0$.
  %For $\theta_0$ as above, we clearly have $r_0 = O(1/m)$; thus the exponential decay rate predicted by~\eqref{eq:solve} should be linear in $m$.
  This is in good agreement with the observed numerics.
  %}}}
  \section{A Scaling Argument and Universal Decay Rates.}\label{sxnScaling}%{{{1
  %\section{Obtaining an algebraic lower bound from a universal exponential lower bound, under no-slip boundary conditions.}\label{sxnScaling}

  Physical intuition suggests that the exponential decay rate in~\eqref{eq:solve} should have some dependence on the size of support of $\theta_0$. % (or at least on $m(\supp(\theta_0))$).
  As we discussed in the introduction, the mixing process can spread around the compactly supported initial data. Whether this has to happen in the mixing process, and whether this
  leads to slowdown in the mixing rate are very interesting open questions. In this section we show that exponential in time lower bound on the decay of the mix norm with the rate in the exponential independent of the initial data would have interesting consequences for mixing in domains with no slip boundaries.
  %In this section we investigate the implications of having an exponential lower bound with rate constant \emph{independent} of $\theta_0$.
 % In this case we show that this lower bound can not be optimal, and will in fact imply a stronger lower bound that decays \emph{algebraically} with respect to time.
 % Our argument is only rigorous for compactly supported initial data and incompressible flows which vanish on the boundary.
 %Let us suppose without loss of generality that the cube $I$ we are working on has side length equal to one.
 %In statement of the Proposition~\ref{boundaries} below, given the initial data $\theta_0$ defined on the whole torus, we will denote $\theta^a_0$ a function
 %defined on $I_a = (0,a)^d$ via $\theta_{0,a}(x) = \theta(x/a,t/a)$ and equal to zero in the rest of $I.$

  \begin{proposition}\label{boundaries}
    Let $I = (0, \ell)^d$ be a cube in $\mathbb R^d$.
    Suppose that there exist $k\in \mathbb R$, $q \in [1, \infty]$ and $c_0 > 0$ %and a positive \emph{algebraic} function $\alpha$
    such that
    %we have a universal lower bound of the form
    \begin{equation}\label{eqnULowerBd}
      \norm{\theta(t)}_{H^{-1}}
	\geq %\alpha\paren[\big]{ \norm{\theta_0}_{W^{k, q}} }
	B(\theta_0)  \exp\paren[\Big]{ \frac{-c_0}{\ell^{d/p}} \int_0^t \norm{\grad u}_{L^p} }
    \end{equation}
    for some $p \in [1, d]$, all incompressible $u$ which vanish on $\del I$, and all initial data $\theta_0 \in C_c^\infty(I).$
    Assume that there exists $\gamma \in \mathbb R$ such that the pre-factor $B(\theta_0)$ satisfies
    $$
      B(A\theta_0)= AB(\theta_0)
      \quad\text{and}\quad
      B(\theta_{0,a})= a^{-\gamma} B(\theta_0),
    $$
    where $\theta_{0,a}(x) = \theta_0(x/a)$ for $x \in (0,a)^d$ and $\theta_{0,a}(x) = 0$ otherwise.
    %Then this bound can not be optimal.

    Then, for any mean zero $\theta_0 \in C_c^\infty(I)$ and any smooth velocity field $u$ such that
    $$
      \limsup_{t \to \infty} \frac{1}{t} \int_0^t \norm{\grad u}_{L^p} <  \infty,
      \quad \grad \cdot u = 0,
      \quad\text{and}\quad u = 0 \text{ on } \del I,
    $$
    the decay of $\norm{\theta(t)}_{H^{-1}}$ as $t \to \infty$ is bounded below by an algebraic function of time.
  \end{proposition}
  \begin{proof}
    We prove this using an elementary scaling argument.
    Without loss of generality, assume $\ell = 1$.
    Let $\theta_0 \in C_c^\infty$ have zero mean and let $f(t) = \norm{\theta(t)}_{H^{-1}}$.
    Our aim is to show that for $t$ large, $f$ decays algebraically with respect to $t$.
    Let $a > 0$, and define
    $$
      I_a = (0, a)^d,\quad
      \eta( x, t ) = \Chi*{I_a} \theta\paren[\big]{ x/a, t/a },
      \quad
      v(x, t) = \Chi*{I_a} u\paren[\big]{ x/a, t/a }.
    $$
    Then $(\eta, v)$ is a solution of~\eqref{eq:1} on the \emph{unscaled} cube $I$.
    Since $\theta_0$ is compactly supported in $I$ and $u = 0$ on $\del I$, we see $\theta(t)$ remains compactly supported in $I_a$ for all $t >0$.
    %Thus
    By our assumption,
    $$
     % \norm{\eta(t)}_{W^{k, q}} = a^{\frac{d}{q} - k} \norm[\big]{\theta\paren[\big]{ \frac{t}{a} } }_{W^{k, q}}
     B(\eta_0) = a^{-\gamma}B(\theta_0)
      \quad\text{and}\quad
      %\norm{\eta(t)}_{W^{\ell, q}} = a^{\frac{d}{q} - k} \norm[\big]{\theta\paren[\big]{ \frac{t}{a} } }.
      \norm{\grad v(t)}_{L^p} = a^{\frac{d}{p} - 1} \norm{\grad u\paren[\big]{t/a } }_{L^p}.
    $$
    Thus the assumed lower bound~\eqref{eqnULowerBd} gives
    \begin{multline*}
      a^{d/2+1} f\paren[\big]{t/a}
      = \norm{\eta(t)}_{H^{-1}}
	\geq %\alpha\paren[\big]{ \norm{\eta_0}_{W^{k, q}} }
       B(\eta_0)
	  \exp\paren[\Big]{ -c_0 \int_0^t \norm{\grad v}_{L^p} }\\
	= %\alpha\paren[\big]{ a^{\frac{d}{q} - k} \norm{\theta_0}_{W^{k, q}} }
       a^{-\gamma}B(\theta_0)
	  \exp\paren[\Big]{ -c_0 a^{\frac{d}{p} - 1} \int_0^t \norm{\grad u\paren[\big]{ s/a }}_{L^p} \, ds }.
    \end{multline*}
    The first equality above follows by duality and scaling.

    Hence, taking $t'=t/a$ gives
    \[ f(t') \geq  a^{-N}B(\theta_0) \exp \left(-c_0 a^{d/p}  \int_0^{t'} \norm{\grad u(s)}_{L^p} \, ds \right), \]
    where $N= d/2+1+\gamma.$
%    $$
%      f(t')
%	\geq
%	  a^{-d/2-1-\gamma}
%	  %\alpha\paren[\big]{ a^{\frac{d}{q} - k} \norm{\theta_0}_{W^{k, q}} }
%B(\theta_0)
%	  \exp\paren[\Big]{ -c_0 a^{d/p} \int_0^{t'} \norm{\grad u(s)}_{L^p} \, ds }.
%    $$
This bound has to be true for every $a>0.$ %Since we assume that $\alpha$ is an algebraic function,
%The bound can be
%written more compactly as
 %may depend on $q,k,$ and $d,$ and $C_1$ also on $\theta_0.$
Maximizing the right hand side in $a$ (and changing $t'$ to $t$), we arrive at
an algebraic lower bound
\begin{gather*}
  f(t) \geq C \left( \frac{N}{\int_0^t \|\nabla u\|_{L^p}\,ds} \right)^{-pN/d}.
  \qedhere
\end{gather*}
%    Now fix $t = 1$ and define a new time variable $t' = 1/a$.
%    Our assumption on $u$ guarantees that if $\frac{d}{p} - 1 \geq 0$ then as $a \to 0$ the exponential factor above is bounded away from $0$.
%    The remainder of the right hand side is algebraic in $a$, and hence in our new time variable $t'$.
%    This shows $f(t')$ is bounded below by an algebraic function of $t'$, finishing the proof.
  \end{proof}

  \section{Acknowledgements.}%{{{1
  The authors would like to thank Charlie Doering for introducing us to this problem and many helpful discussions.

  \bibliographystyle{abbrv}%{{{1
  \bibliography{mixing}
  %}}}1
\end{document}